\documentclass{amsart}

\usepackage{mathrsfs}
\usepackage{mathrsfs}
\usepackage{amsfonts}
\usepackage{amsfonts}
\usepackage{mathrsfs}
\usepackage{amsfonts}
\usepackage{amsfonts}
\usepackage{amsfonts}
\usepackage{mathrsfs}
\usepackage{amsfonts}
\usepackage{amssymb,amsmath}
\usepackage{amsthm}
\usepackage{amsfonts}

\newtheorem{theorem}{Theorem}[section]
\newtheorem{lemma}[theorem]{Lemma}

\theoremstyle{definition}
\newtheorem{definition}[theorem]{Definition}

\newtheorem{proposition}[theorem]{Proposition}

\theoremstyle{remark}
\newtheorem{remark}[theorem]{Remark}

\numberwithin{equation}{section}

\newcommand{\abs}[1]{\lvert#1\rvert}

\begin{document}

\title[Optimal Decay N-S]
 {Optimal Time Decay of Navier-Stokes Equations With Low Regularity Initial Data}

\author[J.X.Jia]{Junxiong Jia}
\address{Department of Mathematics,
Xi'an Jiaotong University,
 Xi'an
710049, China; Beijing Center for Mathematics and Information Interdisciplinary Sciences ( BCMIIS);}
\email{jjx425@gmail.com}

\author[J. Peng]{Jigen Peng}
\address{Department of Mathematics,
Xi'an Jiaotong University,
 Xi'an
710049, China; Beijing Center for Mathematics and Information Interdisciplinary Sciences ( BCMIIS);}
\email{jgpeng@mail.xjtu.edu.cn}








\begin{abstract}
In this paper, we study the optimal time decay rate of isentropic Navier-Stokes equations under the low regularity assumptions about initial data.
In the previous works about optimal time decay rate, the initial data need to be small in $H^{[N/2]+2}(\mathbb{R}^{N})$.
Our work combined negative Besov space estimates and the conventional energy estimates in Besov space framework which is developed by R. Danchin.
Though our methods, we can get optimal time decay rate with initial data just small in $\dot{B}^{N/2-1, N/2+1} \cap \dot{B}^{N/2-1, N/2}$ and belong to some negative Besov
space(need not to be small). Finally, combining the recent results in \cite{zhang2014} with our methods, we can only need the initial data to be
small in homogeneous Besov space $\dot{B}^{N/2-2, N/2} \cap \dot{B}^{N/2-1}$ to get the optimal time decay rate in space $L^{2}$.
\end{abstract}

\maketitle


\section{Introduction and main results}

In this paper, we consider the large time behavior of global strong solutions to the initial value problem for the compressible Navier-Stokes system
\begin{align}\label{originalsystem}
\begin{split}
\begin{cases}
\partial_{t}\rho + \mathrm{div}(\rho u) = 0,    \\
\partial_{t}(\rho u) + \mathrm{div}(\rho u \otimes u) + \nabla P(\rho) - \mu \nabla u - (\lambda + \mu)\nabla \mathrm{div} u = 0,   \\
(\rho, u)|_{t = 0} = (\rho_{0}, u_{0}).
\end{cases}
\end{split}
\end{align}
The variables are the density $\rho > 0$, the velocity $u$. Furthermore, $P = P(\rho)$ is the pressure function. The viscosity coefficient satisfy
$\mu > 0$ and $\lambda + 2\mu > 0$.
$\bar{\rho}$ is a constant and, for simplicity, we can take $\bar{\rho} = 1$ in the present paper.

There are many results on the problem of long time behavior of global solutions to the compressible Navier-Stokes equations. For multi-dimensional
Navier-Stokes equations, the $H^{s}$ global existence and time-decay rate of strong solutions are obtained in whole space first by
Matsumura and Nishida \cite{matsumura1979, matsumura1980} and the optimal $L^{p} (p \geq 2)$ decay rate is established by Ponce \cite{ponce1985}.
To conclude, the optimal $L^{2}$ time decay rate in three dimension is established as
\begin{align}\label{classical}
\begin{split}
\| (\rho - \bar{\rho}, u)(t) \|_{L^{2}(\mathbb{R}^{3})} \leq C (1+t)^{-\frac{3}{4}}
\end{split}
\end{align}
with $(\bar{\rho}, 0)$ to be the constant state and the initial data to be a small perturbation in Sobolev space $H^{l}$ with $l \geq 3$.
In \cite{yanguo2012}, the author developed a general energy method for proving the optimal time decay rates of the solutions.
They removed the small requirement about the initial data in $L^{1}(\mathbb{R}^{3})$ norm, however, they need the initial data
belong to the negative sobolev space $H^{-s}$ with $s \in [0,3/2)$. Their method is so interesting that it needs not to analyze the Green's matrix
of linearized system.
H. Li and T. Zhang \cite{HailiangLi2009} introduced the negative Besov space for the initial data. More specifically, they need
\begin{align*}
\|(\rho_{0} - \bar{\rho}, u_{0})\|_{H^{l}(\mathbb{R}^3) \cap \dot{B}_{1,\infty}^{-s}(\mathbb{R}^{3})} \quad l \geq 4, \,\, s\in [0,1]
\end{align*}
small enough. Then they can obtain the faster time decay rate than the classical results (\ref{classical}) as follows
\begin{align}\label{HailiangLi}
\begin{split}
\| (\rho - \bar{\rho}, u)(t) \|_{L^{2}(\mathbb{R}^{3})} \leq C (1+t)^{-\frac{3}{4}-\frac{s}{2}}.
\end{split}
\end{align}

Time decay in the energy space imply the dissipative properties of the compressible viscous fluids, however do not describe the information
of wave propagation. To understand the wave propagation of compressible viscous in multi-dimension, Hoff-Zumbrun \cite{hoff1995,hoff1997} consider the
Green's function of the compressible isentropic Navier-Stokes equations with the artificial viscosity and show the convergence of global solution
to diffusive waves. Liu-Wang \cite{liutaiping1998} investigate carefully the properties of Green's function for isentropic Navier-Stokes system and
show an interesting convergence of global solution to the diffusive waves with the optimal time-decay rate in odd dimension where the important phenomena of
the weaker Huygens' principle is also justified due to the dispersion effects of compressible viscous fluids in multidimensional odd space.
This is generalized to the full system later in \cite{DLLi} where the wave motions of other types are also introduced. Therein, the optimal
$L^{\infty}$ time-decay rate in three dimensions is
\begin{align}\label{Linftydecay}
\|(\rho(t)-\bar{\rho} , u(t))\|_{L^{\infty}(\mathbb{R}^{3})} \leq C (1+t)^{-3/2}.
\end{align}
The same decay property also proved in the exterior domain \cite{Kobayashi2002,Kobayashi1999}, the half space \cite{Kagei2002,Kagei2005}.

Inspired by the previous works, we also introduced the negative Besov space. Compared to the above work, our work highly reduced the
requirements about the regularity of the initial data and through our proof we revealed an important difference about critical Besov space framework
between compressible field and incompressible field. Below, we give our results specifically.
\begin{theorem}\label{mainth}
Assume that $(\rho_{0} - \bar{\rho}, u_{0}) \in \tilde{B}^{N/2-1,N/2+1} \times (\tilde{B}^{N/2-1,N/2})^{N}$ for an integer $N \geq 2$.
Then there exists a constant $\alpha_{0} > 0$ such that if
\begin{align}
\|\rho_{0} - \bar{\rho}\|_{\tilde{B}^{N/2-1,N/2+1}} + \|u_{0}\|_{\tilde{B}^{N/2-1,N/2}} \leq \alpha_{0},
\end{align}
then the problem (\ref{originalsystem}) admits a unique global solution $(\rho, u)$ satisfying that for all $t \geq 0$,
\begin{align}
\begin{split}
& \|\rho - \bar{\rho}\|_{C(\mathbb{R}^{+}, \tilde{B}^{N/2-1, N/2+1})} + \|u\|_{C(\mathbb{R}^{+}, \tilde{B}^{N/2-1, N/2})}  \\
& \quad\quad\quad\quad\quad\quad\quad + \|\rho - \bar{\rho}\|_{L^{1}(\mathbb{R}^{+}, \dot{B}^{N/2+1})} + \|u\|_{L^{1}(\mathbb{R}^{+}, \tilde{B}^{N/2+1, N/2+2})} \leq C \alpha_{0}.
\end{split}
\end{align}
If further, $(\rho_{0} - \bar{\rho}, u_{0}) \in \dot{B}_{2,\infty}^{-s}$ for some $s \in [0, N/2]$, then for all $t \geq 0$,
\begin{align}
\|\rho(t) - \bar{\rho}\|_{\dot{B}_{2,\infty}^{-s}} + \|u(t)\|_{\dot{B}_{2,\infty}^{-s}} \leq C
\end{align}
and
\begin{align}\label{decayresult}
\|\rho(t) - \bar{\rho}\|_{\dot{B}^{\ell}} + \|u(t)\|_{\dot{B}^{\ell}} \leq C(1+t)^{-\frac{\min(\ell,N/2-1)+s}{2}} \quad \text{for } -s < \ell \leq N/2.
\end{align}
\end{theorem}

\begin{remark}
For $N = 3$ and $s = \frac{3}{2}$, taking $\ell = 0$ in (\ref{decayresult}), we will have
\begin{align*}
\|\rho(t) - \bar{\rho}\|_{L^{2}} + \|u(t)\|_{L^{2}} & \leq \|\rho(t) - \bar{\rho}\|_{\dot{B}^{0}} + \|u(t)\|_{\dot{B}^{0}}  \\
& \leq C (1+t)^{-\frac{3}{4}},
\end{align*}
which coincide with the classical results (\ref{classical}). However, we need not the $L^{1}$, $H^{3}$ norm of initial data to be small.
\end{remark}

\begin{remark}
Taking $N = 3$, $s = \frac{3}{2}$ and $\ell = \frac{3}{2}$, we can obtain the $L^{\infty}$-time decay for 3-D problem as follows
\begin{align*}
\|\rho(t) - \bar{\rho}\|_{L^{\infty}} + \|u(t)\|_{L^{\infty}} & \leq C \|\rho(t) - \bar{\rho}\|_{\tilde{B}^{1/2, 5/2}} + \|u(t)\|_{\tilde{B}^{1/2,3/2}} \\
& \leq C (1+t)^{-1}.
\end{align*}
This time decay rate is a little slower than the optimal $L^{\infty}$-time decay rate. How to get optimal $L^{\infty}$-time decay rate
under the low regularity assumptions on initial data is our future aim.
\end{remark}

\begin{remark}
The methods we used here seems can be generalized to a lot of other fluid dynamic models such as Navier-Stokes Poisson, magnetohydrodynamic flows.
This will be left to the future work.
\end{remark}

\begin{remark}
Recently, C. Wang, W. Wang and Z. Zhang\cite{zhang2014} construct global solution for Navier-Stokes equations allow the initial density and
velocity both have large oscillation. Combining their results with our proof methods, we can get the following results.
\begin{theorem}\label{tuilun}
Suppose the equilibrium state $\bar{\rho} > 0$ and dimension $N = 3$. Let $c_{0}$ to be a small constant,
$P_{+} = \sup_{c_{0}/4 \leq \rho \leq 4 c_{0}^{-1}} |P^{(k)}(\rho)|$ and $P_{-} = \inf_{c_{0}/4 \leq \rho \leq 4 c_{0}^{-1}}|P'(\rho)|$.
Assume that the initial data $(\rho_{0}, u_{0})$ satisfies
\begin{align*}
\rho_{0} - \bar{\rho} \in H^{s} \cap \dot{B}_{2,1}^{1/2,3/2}, \quad c_{0} \leq \rho_{0} \leq c_{0}^{-1}, \quad
u_{0} \in H^{s-1} \cap \dot{B}_{2,1}^{1/2},
\end{align*}
where $s \geq 3$. There exist two constant $c_{1} = c_{1}(\lambda, \mu, c_{0}, \bar{\rho}, P_{+}, P_{-})$
and $c_{2} = c_{2}(\lambda, \mu, c_{0}, \bar{\rho}, P_{+}, P_{-})$ such that if
\begin{align*}
& \|\rho_{0} - \bar{\rho}\|_{\dot{B}_{2,1}^{1/2,3/2}} \leq c_{1}, \quad \|u_{0}\|_{\dot{B}_{2,1}^{1/2}} \leq \frac{c_{2}}{1 + \|\rho_{0} - \bar{\rho}\|_{\dot{B}_{2,1}^{3/2}}},   \\
& \|u_{0}\|_{\dot{H}^{-\delta}} \leq \frac{c_{2}}{(1+\|\rho_{0}-\bar{\rho}\|_{\dot{B}_{2,1}^{3/2}})(1+\|\rho_{0}-\bar{\rho}\|_{H^{2}}^{8} + \|u_{0}\|_{L^{2}}^{4/3})},
\end{align*}
for some $\delta \in \left( -\frac{1}{2}, \frac{3}{2} \right)$, then there exists a unique global solution $(\rho, u)$ to the (\ref{originalsystem})
satisfying
\begin{align*}
& \rho \geq \frac{c_{0}}{4}, \quad \rho - \bar{\rho} \in C([0, + \infty); H^{s}), \\
& u \in C([0, +\infty);H^{s-1})\cap L^{2}(0,T;H^{s}),\quad \text{for any} \,\, T > 0.
\end{align*}
In addition, if $c_{1}$, $c_{2}$ small enough, $(\rho_{0}-\bar{\rho}, u_{0}) \in \dot{B}_{2,\infty}^{-s}$ for some $s \in \left[ 0, \frac{3}{2} \right]$,
then for all $t \geq 0$, we have
\begin{align*}
\|\rho(t) - \bar{\rho}\|_{\dot{B}_{2,\infty}^{-s}} + \|u(t)\|_{\dot{B}_{2,\infty}^{-s}} \leq C,
\end{align*}
and
\begin{align*}
\|\rho(t) - \bar{\rho}\|_{\dot{B}_{2,1}^{\ell}} + \|u(t)\|_{\dot{B}_{2,1}^{\ell}} \leq C (1+t)^{-\frac{\ell +s}{2}},
\end{align*}
for $-s < \ell \leq \frac{1}{2}$, where $C > 0$ is a positive constant depending on the initial data.
\end{theorem}

The proof of this Theorem is similar to the proof of Theorem \ref{mainth}, so we only give Remark \ref{tuilun1}, Remark \ref{tuilun2}
and Remark \ref{tuilun3} in the following to indicate how to get this result.
\end{remark}

The present paper is structured as follows. In section 2, we will give some basic knowledge about homogeneous Besov space. Then in section 3,
we get a differential type inequality about a hyperbolic parabolic system comes from compressible Navier-stokes equations. In section 4, we proved
that Navier-Stokes equations conserve the navigate Besov norms. Theorem \ref{mainth} will be proved by using interpolations of Beosv space in
section 5.

$\mathbf{Notation}$. Throughout the paper, $c$ and $C$ stands for a ``harmless'' constant, and we sometimes use the notation $A \lesssim B$ as an
equivalent of $A \leq C B$. The notation $A \thickapprox B$ means that $A \lesssim B$ and $B \lesssim A$.

\section{Preliminaries}

In this section, we will give some basic knowledge about Besov space, which can be found in \cite{danchin book}.
The homogeneous Littlewood-Paley decomposition relies upon a dyadic partition of unity. We can use for instance any $\phi \in C^{\infty}(\mathbb{R}^{N})$,
supported in $\mathcal{C} := \{\xi \in \mathbb{R}^{N}, 3/4 \leq \abs{\xi} \leq 8/3 \}$ such that
\begin{align*}
\sum_{q \in \mathbb{Z}} \phi(2^{-q} \xi) = 1 \quad \mathrm{if} \quad \xi \neq 0.
\end{align*}
Denote $h = \mathcal{F}^{-1} \phi$, we then define the dyadic blocks as follows
\begin{align*}
\Delta_{q} u := \phi(2^{-q}D)u = 2^{qN} \int_{\mathbb{R}^{N}} h(2^{q}y)u(x-y) \, dy, \quad \mathrm{and} \quad S_{q}u = \sum_{k\leq q-1} \Delta_{k}u.
\end{align*}
The formal decomposition
\begin{align*}
u = \sum_{q\in \mathbb{Z}} \Delta_{q} u
\end{align*}
is called homogeneous Littlewood-Paley decomposition.
The above dyadic decomposition has nice properties of quasi-orthogonality: with our choice of $\phi$, we have
\begin{align*}
& \Delta_{k}\Delta_{q}u = 0 \quad \mathrm{if} \quad \abs{k-q} \geq 2, \\
& \Delta_{k}(S_{q-1}\Delta_{q}u) = 0 \quad \mathrm{if} \quad \abs{k-q} \geq 5.
\end{align*}

Let us now introduce the homogeneous Besov space.
\begin{definition}
We denote by $\mathcal{S}_{h}'$ the space of tempered distributions $u$ such that
\begin{align*}
\lim_{j \rightarrow -\infty} S_{j}u = 0 \quad \mathrm{in} \quad \mathcal{S}'.
\end{align*}
\end{definition}

\begin{definition}
Let $s$ be a real number and $(p,r)$ be in $[1, \infty]^{2}$. The homogeneous Besov space $\dot{B}_{p,r}^{s}$ consists of distributions
$u$ in $\mathcal{S}_{h}'$ such that
\begin{align*}
\|u\|_{\dot{B}_{p,r}^{s}} := \left( \sum_{j \in \mathbb{Z}} 2^{rjs} \|\Delta_{j}u\|_{L^{p}}^{r} \right)^{1/r} < +\infty.
\end{align*}
\end{definition}
From now on, we denote $\dot{B}_{p,1}^{s}$ as $\dot{B}_{p}^{s}$ and the notation $\dot{B}^{s}$ will stand for $\dot{B}_{2,1}^{s}$.

The study of non stationary PDE's usually requires spaces of type $L_{T}^{r}(X) := L^{r}(0,T; X)$ for appropriate Banach spaces $X$.
In our case, we expect $X$ to be a Besov spaces, so that it is natural to localize the equations through Littlewood-Paley decomposition.
We then get estimates for each dyadic block and perform integration in time. However, in doing so, we obtain bounds in spaces which are not
of type $L^{r}(0,T; \dot{B}^{s})$. This approach was initiated in \cite{chemin} and naturally leads to the following definitions:
\begin{definition}
Let $(r,p) \in [1, +\infty]^{2}$, $T \in (0, +\infty]$ and $s\in \mathbb{R}$. We set
\begin{align*}
\|u\|_{\tilde{L}_{T}^{r}(\dot{B}^{s})} := \sum_{q \in \mathbb{Z}} 2^{qs} \left( \int_{0}^{T} \|\Delta_{q}u(t)\|_{L^{2}}^{r} \, dt \right)^{1/r}
\end{align*}
and
\begin{align*}
\tilde{L}_{T}^{r}(\dot{B}^{s}) := \left\{ u \in L_{T}^{r}(\dot{B}^{s}), \|u\|_{\tilde{L}_{T}^{r}(\dot{B}^{s})}< +\infty \right\}.
\end{align*}
\end{definition}
Owing to Minkowski inequality, we have $\tilde{L}_{T}^{r}(\dot{B}^{s}) \hookrightarrow L_{T}^{r}(\dot{B}^{s})$. That embedding is strict in general
if $r>1$. We will denote by $\tilde{C}_{T}(\dot{B}^{s})$ the set of function $u$ belonging to
$\tilde{L}_{T}^{\infty}(\dot{B}^{s})\cap C([0,T]; \dot{B}^{s})$.

Another important space is the bybrid Besov space, we give the definitions and collect some properties.
\begin{definition}
Let $s,t \in \mathbb{R}$. We set
\begin{align*}
\|u\|_{\tilde{B}_{p}^{s,t}} := \sum_{q \leq 0} 2^{qs} \|\Delta_{q}u\|_{L^{p}} + \sum_{q > 0}2^{qt} \|\Delta_{q}u\|_{L^{p}}.
\end{align*}
and
\begin{align*}
\tilde{B}_{p}^{s,t}(\mathbb{R}^{N}) := \left\{ u \in \mathcal{S}_{h}'(\mathbb{R}^{N}), \|u\|_{\tilde{B}_{p}^{s,t}} < +\infty \right\}.
\end{align*}
\end{definition}

\begin{lemma}

1) We have $\tilde{B}_{p}^{s,s} = \dot{B}_{p}^{s}$.

2) If $s\leq t$ then $\tilde{B}_{p}^{s,t} = \dot{B}_{p}^{s} \cap \dot{B}_{p}^{t}$. Otherwise, $\tilde{B}_{p}^{s,t} = \dot{B}_{p}^{s} + \dot{B}_{p}^{t}$.

3) The space $\tilde{B}_{p}^{0,s}$ coincide with the usual inhomogeneous Besov space.

4) If $s_{1} \leq s_{2}$ and $t_{1} \geq t_{2}$ then $\tilde{B}_{p}^{s_{1},t_{1}} \hookrightarrow \tilde{B}_{p}^{s_{2},t_{2}}$.
\end{lemma}

Throughout the paper, we shall use some paradifferential calculus. It is a nice way to define a generalized product between
distributions which is continuous in functional spaces where the usual product does not make sense. The paraproduct between $u$ and $v$
is defined by
\begin{align*}
T_{u}v := \sum_{q \in \mathbb{Z}}S_{q-1}u \Delta_{q}v.
\end{align*}
We thus have the following formal decomposition:
\begin{align*}
uv = T_{u}v + T_{v}u + R(u,v),
\end{align*}
with
\begin{align*}
R(u,v) := \sum_{q\in \mathbb{Z}} \Delta_{q}u \tilde{\Delta}_{q}v, \quad
\tilde{\Delta}_{q} := \Delta_{q-1} + \Delta_{q} + \Delta_{q+1}.
\end{align*}
We will sometimes use the notation $T_{u}'v := T_{u}v + R(u,v)$.

In the last part of this section, we collect some useful lemmas which will be used in the sequel.

\begin{lemma}\label{inter1}\cite{danchin book}
A constant $C$ exists which satisfies the following properties. If $s_{1}$ and $s_{2}$ are real numbers such that $s_{1} < s_{2}$ and $\theta \in (0,1)$, then we have, for any $(p,r) \in [1,\infty]^{2}$ and any $u \in \mathcal{S}'_{h}$,
\begin{align*}
& \|u\|_{\dot{B}_{p,r}^{\theta s_{1} + (1-\theta)s_{2}}} \leq \|u\|_{\dot{B}_{p,r}^{s_{1}}}^{\theta}
\|u\|_{\dot{B}_{p,r}^{s_{2}}}^{1-\theta},	\\
& \|u\|_{\dot{B}_{p,1}^{\theta s_{1} + (1-\theta)s_{2}}} \leq \frac{C}{s_{2}-s_{1}} \left( \frac{1}{\theta} + \frac{1}{1-\theta}\right)\|u\|_{\dot{B}_{p,\infty}^{s_{1}}}^{\theta}\|u\|_{\dot{B}_{p,\infty}^{s_{2}}}^{1-\theta}.
\end{align*}
\end{lemma}

\begin{lemma}\label{inter2}\cite{danchin book}
There exists a constant $C$ such that for any real number $s$ and any $(p,r)\in [1,\infty]^{2}$, we
have, for any $(u,v)\in L^{\infty} \times \dot{B}_{p,r}^{s}$,
\begin{align*}
\|T_{u}v\|_{\dot{B}_{p,r}^{s}} \leq C \|u\|_{L^{\infty}} \|v\|_{\dot{B}_{p,r}^{s}}.
\end{align*}
Moreover, for any $(s,t) \in \mathbb{R} \times (-\infty,0)$ and any $(p, r_{1}, r_{2}) \in [1,
\infty]^{3}$we have, for any $(u,v) \in \dot{B}_{\infty,r_{1}}^{t} \times \dot{B}_{p,r_{2}}^{s}$,
\begin{align*}
\|T_{u}v\|_{\dot{B}_{p,r}^{s+t}} \leq \frac{C}{t} \|u\|_{\dot{B}_{\infty,r_{1}}^{t}}\|v\|_{\dot{B}_{p,r_{2}}^{s}}
\quad \text{with} \quad \frac{1}{r} = \mathrm{min}\left\{ 1, \frac{1}{r_{1}} + \frac{1}{r_{2}} \right\}
\end{align*}
\end{lemma}

\begin{remark}\label{inter3}
Using similar methods as in the proof of Lemma \ref{inter2}, we can get the following estimates
\begin{align*}
\|T_{u}v\|_{\dot{B}_{2,\infty}^{-s}} \leq C \|u\|_{\dot{B}_{\infty, 1}^{0}} \|v\|_{\dot{B}_{2,\infty}^{-s}}.
\end{align*}
\end{remark}

\begin{lemma}\label{inter4}\cite{danchin book}
A constant $C$ exists which satisfies the following inequalities. Let $(s_{1}, s_{2})$ be in $\mathbb{R}^{2}$ and $(p_{1}, p_{2}, r_{1}, r_{2})$ be in $[1,\infty]^{4}$. Assume that
\begin{align*}
\frac{1}{p} = \frac{1}{p_{1}} + \frac{1}{p_{2}} \leq 1 \quad \text{and} \quad
\frac{1}{r} = \frac{1}{r_{1}} + \frac{1}{r_{2}} \leq 1.
\end{align*}
If $s_{1} + s_{2}$ is positive, then we have, for any $(u, v)$ in $\dot{B}_{p_{1},r_{1}}^{s_{1}} \times \dot{B}_{p_{2},r_{2}}^{s_{2}}$,
\begin{align*}
\|R(u,v)\|_{\dot{B}_{p,r}^{s_{1}+s_{2}}} \leq \frac{C}{s_{1}+s_{2}} \|u\|_{\dot{B}_{p_{1},r_{1}}^{s_{1}}}
\|v\|_{\dot{B}_{p_{2},r_{2}}^{s_{2}}}.
\end{align*}
When $r = 1$ and $s_{1} + s_{2} \geq 0$, we have, for any $(u,v)$ in $\dot{B}_{p_{1},r_{1}}^{s_{1}}\times \dot{B}_{p_{2},r_{2}}^{s_{2}}$,
\begin{align*}
\|R(u,v)\|_{\dot{B}_{p,\infty}^{s_{1}+s_{2}}} \leq C \|u\|_{\dot{B}_{p_{1},r_{1}}^{s_{1}}} \|v\|_{\dot{B}_{p_{2},r_{2}}^{s_{2}}}.
\end{align*}
\end{lemma}

\begin{lemma} \cite{danchin full}
\label{simple convection lemma}
Let $F$ be an homogeneous smooth function of degree $m$. Suppose that $-\frac{N}{2} < s_{1},t_{1},s_{2},t_{2} \leq 1+\frac{N}{2}$.
The following estimate hold:
\begin{align}
\begin{split}
|(F(D)\Delta_{q} & (u\cdot \nabla a)|F(D)\Delta_{q}a)| \\
& \lesssim \, \alpha_{q} 2^{-q(\tilde{\phi}^{s_{1},s_{2}}(q)-m)}\|u\|_{\dot{B}^{N/2+1}} \|a\|_{\tilde{B}^{s_{1},s_{2}}} \|F(D)\Delta_{q}a\|_{L^{2}},
\end{split}
\end{align}
with $\sum_{q \in \mathbb{Z}} \alpha_{q} \leq 1$ and
\begin{align*}
\tilde{\phi}^{\alpha,\beta}(q) = \alpha \quad \mathrm{if} \,\, q \leq 0, \\
\tilde{\phi}^{\alpha,\beta}(q) = \beta \quad \mathrm{if} \,\, q \geq 1.
\end{align*}
\end{lemma}

\section{Estimates about hyperbolic parabolic system}

This part is inspired by R. Danchin's seminal work \cite{danchin NS}, however, we need a differential type estimates different to the integral type estimates proved by R. Danchin.
Firstly, let us give some equivalent forms of the systems (\ref{originalsystem}). Denoting $a = \rho - 1$, supposing $P'(1) = 1$, $N \geq 2$ stands for dimension, then we have
\begin{align}\label{2system}
\begin{split}
\begin{cases}
\partial_{t}a + u \cdot \nabla a + \mathrm{div} u = - a \mathrm{div} u,     \\
\partial_{t}u + u \cdot \nabla u - \mu \Delta u - (\lambda + \mu)\nabla \mathrm{div} u + \nabla a   \\
\quad\quad\quad\quad\quad\quad\quad\quad\quad\quad = - h(a) \left( \mu \Delta u + (\lambda + \mu)\nabla\mathrm{div}u \right) - f(a)\nabla a,     \\
(a, u)|_{t=0} = (\rho_{0}-1, u_{0}),
\end{cases}
\end{split}
\end{align}
where
\begin{align*}
h(a) = \frac{a}{1+a},\quad f(a) = \frac{P'(1+a)}{1+a} - P'(1).
\end{align*}
Denoting $d = \Lambda^{-1}\mathrm{div} u$ and $\Omega = \Lambda^{-1}\mathrm{curl}u$ with $(\mathrm{curl}z)_{i}^{j} = \partial_{j}z^{i}-\partial_{i}z^{j}$,
we will obtain the following equivalent form
\begin{align}\label{3system}
\begin{split}
\begin{cases}
\partial_{t}a + u \cdot \nabla a + \Lambda d = F,   \\
\partial_{t}d + u \cdot \nabla d - \nu \Delta d - \Lambda a = G,    \\
\partial_{t}\Omega - \mu \Omega = H,    \\
u = - \Lambda^{-1}\nabla d - \Lambda^{-1} \mathrm{div} \Omega,
\end{cases}
\end{split}
\end{align}
where $\nu = \lambda + 2\mu$ and
\begin{align*}
& F = - a \mathrm{div} u,     \\
& G = u \cdot\nabla d + \Lambda^{-1}\mathrm{div}\left( -u \cdot\nabla u - h(a) (\mu \Delta u + (\lambda + \mu)\nabla \mathrm{div}u ) - f(a)\nabla a \right),    \\
& H = \Lambda^{-1}\mathrm{curl}\left( - u\cdot\nabla u - h(a)(\mu \Delta u + (\lambda + \mu)\nabla \mathrm{div}u ) \right).
\end{align*}
The equivalence is not so obvious, considering the proof can be found in \cite{danchin full,xianpenghu2011}, we will omit it.
Denote
\begin{align*}
\label{space frame}
E^{s} = & L^{1}(\mathbb{R}^{+}; \dot{B}^{s} \times \tilde{B}^{s+1, s+2}) \cap C_{b}(\mathbb{R}^{+}; \tilde{B}^{s-1, s+1} \times \tilde{B}^{s-1, s}).
\end{align*}
Then, we give the main results in this section.
\begin{proposition}\label{hyperbolicparabolic}
Let $(a, d, \Omega)$ be solution of (\ref{3system}) on $[0, T)$, $2 - N/2 < s \leq 1 + N/2$, $V(t) = \int_{0}^{t} \|u(\tau)\|_{\dot{B}^{N/2+1}} \,d\tau$,
and $V(\infty)$ is bounded.
Then there exists a functional
\begin{align}
\mathcal{E}_{s}(a,d,\Omega) \approx \|a(t)\|_{\tilde{B}^{s-2,s}} + \|d(t)\|_{\tilde{B}^{s-2,s-1}} + \|\Omega(t)\|_{\tilde{B}^{s-2,s-1}},
\end{align}
such that
\begin{align}\label{diff}
\begin{split}
& \frac{d}{dt} \left( e^{-C V(t)} \mathcal{E}_{s}(a,d,\Omega) \right) + c e^{-C V(t)} \bigg( \|a(t)\|_{\dot{B}^{s}} + \|d(t)\|_{\tilde{B}^{s,s+1}}  \\
& \quad\quad + \|\Omega(t)\|_{\tilde{B}^{s,s+1}} \bigg) \leq C e^{-CV(t)} \left( \|F(t)\|_{\tilde{B}^{s-2,s}} + \|(G(t),H(t))\|_{\tilde{B}^{s-2,s-1}} \right),
\end{split}
\end{align}
where $C$ depends on $\mu, \lambda, N, s$.
\end{proposition}
\begin{proof}
Let $K > 0$ to be a constant which will be specified later, define
\begin{align*}
& (\tilde{a}, \tilde{d}, \tilde{\Omega}) = e^{-KV(t)} (a, d, \Omega),  \\
& (\tilde{F}, \tilde{G}) = e^{-KV(t)} (F, G).
\end{align*}
Applying the operator $\Delta_{q}$ to the first two equations in (\ref{3system}), we infer that $(\Delta_{q}\tilde{a}, \Delta_{q}\tilde{d})$ satisfies
\begin{align}\label{localsystem}
\begin{split}
& \partial_{t}\Delta_{q}\tilde{a} + \Delta_{q} (u\cdot\nabla \tilde{a}) + \Lambda \Delta_{q}\tilde{d} = \Delta_{q}\tilde{F} - KV'(t)\Delta_{q}\tilde{a},  \\
& \partial_{t}\Delta_{q}\tilde{d} + \Delta_{q} (u\cdot\nabla \tilde{d}) - \nu \Delta \Delta_{q}\tilde{d} - \Lambda \Delta_{q} \tilde{a} = \Delta\tilde{G}
- KV'(t) \Delta_{q}\tilde{d}.
\end{split}
\end{align}
$\mathbf{Step \,\, 1:low\,\,\, frequencies.}$
Suppose that $q \leq q_{0}$ and $q_{0}$ is a constant specified in Step 2, define
\begin{align*}
f_{q}^{2} = \|\Delta_{q}\tilde{a}\|_{L^{2}}^{2} + \|\Delta_{q}\tilde{d}\|_{L^{2}}^{2} - r (\Lambda \Delta_{q}\tilde{a} | \Delta_{q}\tilde{d}).
\end{align*}
Let $a_{0} = \frac{8}{3}$ and $ra_{0} < 1$.
According to the H\"{o}lder's inequality, we have
\begin{align*}
|(\Lambda \Delta_{q}\tilde{a} | \Delta_{q}\tilde{d})| \leq a_{0} \left( \|\Delta_{q}\tilde{a}\|_{L^{2}}^{2} + \|\Delta\tilde{d}\|_{L^{2}}^{2} \right).
\end{align*}
Hence, we know that
\begin{align*}
f_{q}^{2} \approx \|\Delta_{q}\tilde{a}\|_{L^{2}}^{2} + \|\Delta_{q}\tilde{d}\|_{L^{2}}^{2}
\end{align*}
Using standard energy methods, we can obtain the following basic energy estimates
\begin{align}\label{basic1}
\begin{split}
& \frac{1}{2}\frac{d}{dt}\|\Delta_{q}\tilde{a}\|_{L^{2}}^{2} + (\Delta_{q}(u\cdot\nabla\tilde{a}) | \Delta_{q}\tilde{a}) + (\Lambda\Delta_{q}\tilde{d} | \Delta_{q}\tilde{a}) \\
& \quad\quad\quad\quad\quad\quad\quad\quad\quad\quad\quad = (\Delta_{q}\tilde{F} | \Delta_{q}\tilde{a}) - KV'(t) \|\Delta_{q}\tilde{a}\|_{L^{2}}^{2},
\end{split}
\end{align}
\begin{align}\label{basic2}
\begin{split}
& \frac{1}{2}\frac{d}{dt}\|\Delta_{q}\tilde{d}\|_{L^{2}}^{2} + (\Delta_{q}(u\cdot\nabla\tilde{d}) | \Delta_{q}\tilde{d}) - \nu (\Delta\Delta_{q}\tilde{d} | \Delta_{q}\tilde{d})  \\
& \quad\quad\quad\quad - (\Lambda\Delta_{q}\tilde{a} | \Delta_{q}\tilde{d}) = (\Delta_{q}\tilde{G} | \Delta_{q}\tilde{d}) - KV'(t)\|\Delta_{q}\tilde{d}\|_{L^{2}}^{2},  \\
\end{split}
\end{align}
\begin{align}\label{basic3}
\begin{split}
& \frac{d}{dt}(\Lambda\Delta_{q}\tilde{a} | \Delta\tilde{d}) + \|\Lambda\Delta_{q}\tilde{d}\|_{L^{2}}^{2} - \|\Lambda\Delta_{q}\tilde{a}\|_{L^{2}}^{2} +
(\Lambda\Delta_{q}(u\cdot\nabla\tilde{a}) | \Delta_{q}\tilde{d})    \\
& \quad\quad\quad\quad + (\Delta_{q}(u\cdot\nabla\tilde{d}) | \Lambda\Delta_{q}\tilde{a}) + \mu (\Lambda^{2}\Delta_{q}\tilde{d} | \Lambda\Delta_{q}\tilde{a})
= (\Lambda\Delta_{q}\tilde{F} | \Delta\tilde{d})    \\
& \quad\quad\quad\quad + (\Delta\tilde{G} | \Lambda\Delta_{q}\tilde{a}) - 2KV'(t)(\Lambda\Delta_{q}\tilde{a} | \Delta_{q}\tilde{d}).
\end{split}
\end{align}
Performing the following calculation
\begin{align*}
(\ref{basic1}) + (\ref{basic2}) - \frac{r}{2} (\ref{basic3}),
\end{align*}
we have
\begin{align*}
\frac{1}{2}\frac{d}{dt}f_{q}^{2} + \mathcal{M}_{1} = \mathcal{M}_{2} + \mathcal{M}_{3} - KV'(t)f_{q}^{2},
\end{align*}
where
\begin{align*}
\mathcal{M}_{1} = \left(\nu - \frac{r}{2}\right)\|\Lambda\Delta_{q}\tilde{d}\|_{L^{2}}^{2} + \frac{r}{2}\|\Lambda\Delta_{q}\tilde{a}\|_{L^{2}}^{2}
- \frac{\nu r}{2} (\Lambda^{2}\Delta_{q}\tilde{d} | \Lambda\Delta_{q}\tilde{c}),
\end{align*}
\begin{align*}
\mathcal{M}_{2} = (\Delta_{q}\tilde{F} | \Delta_{q}\tilde{a}) + (\Delta_{q}\tilde{G} | \Delta_{q}\tilde{d}) - \frac{r}{2}(\Lambda\Delta_{q}\tilde{F} | \Delta\tilde{d})
- \frac{r}{2}(\Delta_{q}\tilde{G} | \Lambda\Delta_{q}\tilde{a}),
\end{align*}
\begin{align*}
& \mathcal{M}_{3} = -(\Delta_{q}(u\cdot\nabla\tilde{a}) | \Delta_{q}\tilde{a}) - (\Delta_{q}(u\cdot\nabla\tilde{d}) | \Delta_{q}\tilde{d})
+ \frac{r}{2}(\Lambda\Delta_{q}(u\cdot\nabla\tilde{a}) | \Delta_{q}\tilde{d})   \\
& \quad\quad\quad + \frac{r}{2}(\Delta_{q}(u\cdot\nabla\tilde{d}) | \Lambda\Delta_{q}\tilde{a}).
\end{align*}
Since
\begin{align*}
\frac{r\nu}{2}|(\Lambda^{2}\Delta_{q}\tilde{d} | \Lambda\Delta_{q}\tilde{a})| \leq \frac{r\nu^{2}a_{0}^{2}}{2}\|\Lambda\Delta_{q}\tilde{d}\|_{L^{2}}^{2}
+ \frac{r}{8} \|\Lambda\Delta_{q}\tilde{a}\|_{L^{2}}^{2}.
\end{align*}
Taking $r$ so small such that
\begin{align*}
\nu - \frac{r}{2} - \frac{r\nu^{2}a_{0}^{2}}{2} > 0.
\end{align*}
Then the term $\mathcal{M}_{1}$ can be bounded from below by $c2^{2q}f_{q}^{2}$.
According to H\"{o}lder's inequality, we obtain
\begin{align*}
\mathcal{M}_{2} \lesssim & \, \|\Delta_{q}\tilde{F}\|_{L^{2}} \|\Delta_{q}\tilde{a}\|_{L^{2}} + \|\Delta_{q}\tilde{G}\|_{L^{2}} \|\Delta_{q}\tilde{d}\|_{L^{2}}   \\
& + r \|\Lambda \Delta_{q}\tilde{F}\|_{L^{2}} \|\Delta_{q}\tilde{d}\|_{L^{2}} + r\|\Delta_{q}\tilde{G}\|_{L^{2}} \|\Lambda \Delta_{q}\tilde{a}\|_{L^{2}} \\
\lesssim & \, \alpha_{q} 2^{-q(s-2)} f_{q} \|(\tilde{F}, \tilde{G})\|_{\tilde{B}^{s-2,s} \times \tilde{B}^{s-2,s-1}}.
\end{align*}
Using Lemma \ref{simple convection lemma}  directly, we get the estimates of $\mathcal{M}_{3}$ as follows:
\begin{align*}
\mathcal{M}_{3} \lesssim & \, \alpha_{q} 2^{-q(s-2)}V'(t)\bigg( \|\tilde{a}\|_{\tilde{B}^{s-2,s}} + \|\tilde{d}\|_{\tilde{B}^{s-2,s-1}}  \bigg)
 \bigg( \|\Delta_{q}\tilde{a}\|_{L^{2}}   \\
& + \|\Delta_{q}\tilde{d}\|_{L^{2}}  \bigg)
 + \alpha_{q} V'(t) r \bigg( 2^{-q(s-2)} \|\Lambda \Delta_{q}\tilde{a}\|_{L^{2}} \|\tilde{d}\|_{\tilde{B}^{s-2,s-1}} \\
& + 2^{-q(s-2)}2^{q} \|\tilde{a}\|_{\tilde{B}^{s-2,s}} \|\Delta_{q}\tilde{d}\|_{L^{2}} \bigg)  \\
\lesssim & \, \alpha_{q} 2^{-q(s-2)} V'(t) f_{q} \|(\tilde{a}, \tilde{d})\|_{\tilde{B}^{s-2,s}  \times \tilde{B}^{s-2,s-1}}.
\end{align*}
So for the low frequency part, we finally have
\begin{align}\label{lowfre}
\begin{split}
& \frac{d}{dt}f_{q} + c 2^{2q} f_{q} \leq C \alpha_{q} 2^{-q(s-2)}\Big( \|(\tilde{F}, \tilde{G})\|_{\tilde{B}^{s-2,s}\times\tilde{B}^{s-2,s-1}}   \\
& \quad\quad\quad\quad\quad\quad\quad\quad\quad\quad + V'(t) \|(\tilde{a}, \tilde{d})\|_{\tilde{B}^{s-2,s}\times\tilde{B}^{s-2,s-1}} \Big) - CKV'(t)f_{q}.
\end{split}
\end{align}
$\mathbf{Step \,\, 2:high\,\,\, frequencies.}$
Suppose $q \geq q_{0} + 1$ in this part. Define
\begin{align*}
f_{q}^{2} = \|\Lambda^{2}\Delta_{q}\tilde{a}\|_{L^{2}}^{2} + \frac{2}{\nu^{2}}\|\Lambda \Delta_{q}\tilde{d}\|_{L^{2}}^{2} -
\frac{2}{\nu}(\Lambda^{2}\Delta_{q}\tilde{a} | \Lambda\Delta_{q}\tilde{d}).
\end{align*}
Since
\begin{align*}
\frac{2}{\nu}|(\Lambda^{2}\Delta_{q}\tilde{a} | \Lambda\Delta_{q}\tilde{d})|\leq \frac{1}{2}\|\Lambda^{2}\Delta_{q}\tilde{a}\|_{L^{2}}^{2}
+ \frac{2}{\nu^2} \|\Lambda\Delta_{q}\tilde{d}\|_{L^{2}}^{2},
\end{align*}
we know that
\begin{align*}
f_{q}^{2} \approx \|\Lambda^{2}\Delta_{q}\tilde{a}\|_{L^{2}}^{2} + \|\Lambda\Delta_{q}\tilde{d}\|_{L^{2}}^{2}.
\end{align*}
Similar to the low frequency part, we have the following basic energy estimates.
\begin{align}\label{basic4}
\begin{split}
& \frac{1}{2}\frac{d}{dt}\|\Lambda^{2}\Delta_{q}\tilde{a}\|_{L^{2}}^{2} + (\Lambda^{2}\Delta_{q}(u\cdot\nabla\tilde{a}) | \Lambda^{2}\Delta_{q}\tilde{a}) \\
& \quad\quad\quad\quad + (\Lambda^{3}\Delta_{q}\tilde{d} | \Lambda^{2}\Delta_{q}\tilde{a}) = (\Lambda^{2}\Delta_{q}\tilde{F} | \Lambda^{2}\Delta_{q}\tilde{a})
- KV'(t)\|\Lambda^{2}\Delta_{q}\tilde{a}\|_{L^{2}}^{2},
\end{split}
\end{align}
\begin{align}\label{basic5}
\begin{split}
& \frac{1}{2}\frac{d}{dt}\|\Lambda\Delta_{q}\tilde{d}\|_{L^{2}}^{2} + (\Lambda\Delta_{q}(u\cdot\nabla\tilde{d}) | \Lambda\Delta_{q}\tilde{d}) + \nu \|\Lambda^{2}\Delta_{q}\tilde{a}\|_{L^{2}}^{2}  \\
& \quad\quad\quad\quad - (\Lambda^{2}\Delta_{q}\tilde{a} | \Lambda\Delta_{q}\tilde{d}) = (\Lambda\Delta_{q}\tilde{G} | \Lambda\Delta_{q}\tilde{d})
- KV'(t)\|\Lambda\Delta_{q}\tilde{d}\|_{L^{2}}^{2},
\end{split}
\end{align}
\begin{align}\label{basic6}
\begin{split}
& \frac{d}{dt}(\Lambda^{2}\Delta_{q}\tilde{a} | \Lambda\Delta_{q}\tilde{d}) + \|\Lambda^{2}\Delta_{q}\tilde{d}\|_{L^{2}}^{2} + (\Lambda^{2}\Delta_{q}(u\cdot\nabla\tilde{a}) | \Lambda\Delta_{q}\tilde{d}) \\
& \quad\quad + (\Lambda\Delta_{q}(u\cdot\nabla\tilde{d})|\Lambda^{2}\Delta_{q}\tilde{a}) + \nu (\Lambda^{3}\Delta_{q}\tilde{d} | \Lambda^{2}\Delta_{q}\tilde{a})
- \|\Lambda^{2}\Delta_{q}\tilde{a}\|_{L^{2}}^{2}    \\
& \quad\quad = (\Lambda^{2}\Delta_{q}\tilde{F} | \Lambda\Delta_{q}\tilde{d}) + (\Lambda\Delta_{q}\tilde{G} | \Lambda^{2}\Delta_{q}\tilde{a})
- 2KV'(t)(\Lambda^{2}\Delta_{q}\tilde{a} | \Lambda\Delta_{q}\tilde{d}).
\end{split}
\end{align}
Performing the following calculation
\begin{align*}
(\ref{basic4}) + \frac{2}{\nu^{2}}(\ref{basic5}) - \frac{1}{\nu}(\ref{basic6}),
\end{align*}
we will obtain
\begin{align*}
\frac{1}{2}\frac{d}{dt}f_{q}^{2} + \mathcal{M}_{1} = \mathcal{M}_{2} + \mathcal{M}_{3} - KV'(t)f_{q}^{2},
\end{align*}
where
\begin{align*}
\mathcal{M}_{1} = \frac{1}{\nu} \|\Lambda^{2}\Delta_{q}\tilde{d}\|_{L^{2}}^{2} + \frac{1}{\nu} \|\Lambda^{2}\Delta_{q}\tilde{a}\|_{L^{2}}^{2} -
\frac{2}{\nu^{2}}(\Lambda^{2}\Delta_{q}\tilde{a} | \Lambda\Delta_{q}\tilde{d})
\end{align*}
\begin{align*}
\mathcal{M}_{2} = & (\Lambda^{2}\Delta_{q}\tilde{F} | \Lambda^{2}\Delta_{q}\tilde{a}) + \frac{2}{\nu^{2}}(\Lambda\Delta_{q}\tilde{G} | \Lambda\Delta_{q}\tilde{d})  \\
& -\frac{1}{\nu}(\Lambda^{2}\Delta_{q}\tilde{F} | \Lambda\Delta_{q}\tilde{d}) - \frac{1}{\nu}(\Lambda\Delta_{q}\tilde{G} | \Lambda^{2}\Delta_{q}\tilde{a}),
\end{align*}
\begin{align*}
\mathcal{M}_{3} = & -(\Lambda^{2}\Delta_{q}(u\cdot\nabla\tilde{a}) | \Lambda^{2}\Delta_{q}\tilde{a}) - \frac{2}{\nu^{2}}(\Lambda\Delta_{q}(u\cdot\nabla\tilde{d})|\Lambda\Delta\tilde{d}) \\
& + \frac{1}{\nu} \left( (\Lambda^{2}\Delta_{q}(u\cdot\nabla\tilde{a}) | \Lambda\Delta_{q}\tilde{d}) + (\Lambda\Delta_{q}(u\cdot\nabla\tilde{d}) | \Lambda^{2}\Delta_{q}\tilde{a}) \right).
\end{align*}
We can easily obtain
\begin{align*}
& |(\Lambda\Delta_{q}\tilde{d} | \Lambda^{2}\Delta_{q}\tilde{a})| \leq \frac{\nu}{4}\|\Lambda^{2}\Delta_{q}\tilde{a}\|_{L^{2}}^{2} + \frac{1}{\nu}\|\Lambda\Delta_{q}\tilde{d}\|_{L^{2}}^{2},  \\
& \|\Lambda^{2}\Delta_{q}\tilde{d}\|_{L^{2}}^{2} \geq \frac{9}{16}2^{2(q_{0}+1)} \|\Lambda\Delta_{q}\tilde{d}\|_{L^{2}}^{2}.
\end{align*}
Taking $q_{0}$ large enough such that
\begin{align*}
\frac{1}{\nu}\frac{9}{4}2^{2q_{0}} - \frac{2}{\nu^3} > 0.
\end{align*}
$\mathcal{M}_{1}$ can be bounded from below by $cf_{q}^{2}$.
Similar to the low frequency part, we have
\begin{align*}
& \mathcal{M}_{2} \leq C \alpha_{q} f_{q} 2^{-q(s-2)} \|(\tilde{F}, \tilde{G})\|_{\tilde{B}^{s-2,s} \times \tilde{B}^{s-2,s-1}},  \\
& \mathcal{M}_{3} \leq C \alpha_{q} f_{q} 2^{-q(s-2)} V'(t) \|(\tilde{a}, \tilde{d})\|_{\tilde{B}^{s-2,s-1} \times \tilde{B}^{s-2,s-1}}.
\end{align*}
So finally, for high frequency part, we know that
\begin{align}\label{highfre}
\begin{split}
\frac{d}{dt}f_{q} + c f_{q} \leq & C \alpha_{q} 2^{-q(s-2)} \Big( \|(\tilde{F}, \tilde{G})\|_{\tilde{B}^{s-2,s}\times\tilde{B}^{s-2,s-1}}    \\
& + V'(t)\|(\tilde{a}, \tilde{d})\|_{\tilde{B}^{s-2,s}\times\tilde{B}^{s-2,s-1}} \Big) - KV'(t)f_{q}.
\end{split}
\end{align}
From the above analysis, we deduced that
\begin{align}\label{equi}
2^{q(s-2)}f_{q} \approx 2^{qs} \mathrm{max}(1,2^{-2q})\|\Delta_{q}\tilde{a}\|_{L^{2}} + 2^{q(s-1)}\mathrm{max}(1, 2^{-q})\|\Delta_{q}\tilde{d}\|_{L^{2}}.
\end{align}
$\mathbf{Step \,\, 3:the\,\,\, damping\,\,\, effect.}$
For large enough constant $K$, we have
\begin{align*}
\sum_{q\in \mathbb{Z}} \left( C \alpha_{q}(t)\|(\tilde{a}, \tilde{d})\|_{\tilde{B}^{s-2,s}\times\tilde{B}^{s-2,s-1}} - K 2^{q(s-2)}f_{q}(t) \right) \leq 0.
\end{align*}
So after simple calculations, we have
\begin{align}\label{dampeff}
\begin{split}
\frac{d}{dt}\left( \sum_{q\in \mathbb{Z}} 2^{q(s-2)} f_{q} \right) & + c \|\tilde{a}(t)\|_{\dot{B}^{s}} \\
& + c \sum_{q\in \mathbb{Z}} 2^{q(s-1)} \mathrm{min}(2^{2q}, 1)\mathrm{max}(1, 2^{-q})\|\Delta_{q}\tilde{d}(t)\|_{L^{2}}    \\
& \leq C \|\tilde{F}(t)\|_{\tilde{B}^{s-2,s}} + C \|\tilde{G}\|_{\tilde{B}^{s-2,s-1}}.
\end{split}
\end{align}
$\mathbf{Step \,\, 4:the\,\,\, smoothing\,\,\, effect.}$
In this step, we investigate the smoothing effect on $d$. Due to (\ref{dampeff}), we just need to prove it for high frequencies only.
Let $k_{q}^{2} = \|\Delta_{q}\tilde{d}\|_{L^{2}}^{2}$. We can easily get
\begin{align*}
\frac{1}{2}\frac{d}{dt}k_{q}^{2} + \nu \|\Lambda\Delta_{q}\tilde{d}\|_{L^{2}}^{2} = & (\Delta_{q}\tilde{G} | \Delta_{q}\tilde{d})
+ (\Lambda\Delta_{q}\tilde{a} | \Delta_{q}\tilde{d})    \\
& - (\Delta_{q}(u\cdot\nabla\tilde{d}) | \Delta_{q}\tilde{d}) - KV'(t)k_{q}^{2}
\end{align*}
Choosing $q \geq q_{0} + 1$, we have $\|\Lambda\Delta_{q}\tilde{d}\|_{L^{2}}^{2} \geq \frac{9}{16}2^{2q}\|\Delta_{q}\tilde{d}\|_{L^{2}}^{2}$.
Hence, we get
\begin{align*}
\frac{d}{dt}k_{q} + c 2^{2q} k_{q} \leq C \|\Lambda\Delta_{q}\tilde{a}\|_{L^{2}} + \|\Delta_{q}\tilde{G}\|_{L^{2}}
+ CV'(t) \alpha_{q}2^{-2q(s-1)} \|\tilde{d}\|_{\tilde{B}^{s-2,s-1}}.
\end{align*}
Summing up for $q \geq q_{0} + 1$, we have
\begin{align}\label{smootheff}
\begin{split}
& \frac{d}{dt}\left( \sum_{q\geq q_{0}+1}2^{q(s-1)}k_{q} \right) + \sum_{q \geq q_{0}+1}2^{q(s+1)}\|\Delta_{q}\tilde{d}\|_{L^{2}}   \\
& \leq C \sum_{q \geq q_{0} + 1}2^{qs} \|\Delta_{q}\tilde{a}\|_{L^{2}} + C \|\tilde{G}\|_{\tilde{B}^{s-2,s-1}} + CV'(t)\|\tilde{d}\|_{\tilde{B}^{s-2,s-1}}.
\end{split}
\end{align}
Taking a small positive constant $\epsilon$, and performing the following calculations
\begin{align*}
(\ref{dampeff}) + \epsilon (\ref{smootheff}).
\end{align*}
We will get
\begin{align*}
\frac{d}{dt}e^{-KV(t)}\mathcal{F}_{s}(a, d) + c \|\tilde{a}(t)\|_{\dot{B}^{s}} + & c \|\tilde{d}(t)\|_{\tilde{B}^{s,s+1}} \leq C \|\tilde{F}(t)\|_{\tilde{B}^{s-2,s}} \\
& + C \|\tilde{G}(t)\|_{\tilde{B}^{s-2,s-1}} + CV'(t) \|\tilde{d}(t)\|_{\tilde{B}^{s-2,s-1}}
\end{align*}
where
\begin{align*}
\mathcal{F}_{s}(a, d) = e^{KV(t)} \left( \sum_{q\in \mathbb{Z}}2^{q(s-2)}f_{q} + \epsilon \sum_{q \geq q_{0}+1}2^{q(s-1)}k_{q}\right).
\end{align*}
Obviously, we know that
\begin{align*}
\mathcal{F}_{s}(a, d) \approx \|a(t)\|_{\tilde{B}^{s-2,s}} + \|d(t)\|_{\tilde{B}^{s-2,s-1}}.
\end{align*}
Using the above identity, further more, we can obtain
\begin{align*}
\frac{d}{dt}e^{-KV(t)}\mathcal{F}_{s}(a, d) + c \|\tilde{a}(t)\|_{\dot{B}^{s}} + & c \|\tilde{d}(t)\|_{\tilde{B}^{s,s+1}} \leq C \|\tilde{F}(t)\|_{\tilde{B}^{s-2,s}} \\
& + C \|\tilde{G}(t)\|_{\tilde{B}^{s-2,s-1}} + CV'(t) e^{-KV(t)} \mathcal{F}_{s}(a, d).
\end{align*}
Denote $C+K$ as $C$, we have
\begin{align}\label{a and d}
\begin{split}
\frac{d}{dt}e^{-CV(t)}\mathcal{F}_{s}(a, d) + & c e^{-CV(t)} \left( \|a(t)\|_{\dot{B}^{s}} + \|d(t)\|_{\tilde{B}^{s,s+1}} \right)   \\
& \leq C e^{-CV(t)} \left( \|F(t)\|_{\tilde{B}^{s-2,s}} + \|G(t)\|_{\tilde{B}^{s-2,s-1}} \right).
\end{split}
\end{align}
$\mathbf{Step \,\, 5:the\,\,\, equation\,\,\, of\,\,\, \Omega .}$
$\tilde{\Omega}$ satisfies
\begin{align*}
\partial_{t}\tilde{\Omega} - \mu \Delta\tilde{\Omega} = \tilde{H} - KV'(t)\tilde{\Omega}.
\end{align*}
Localizing the above equation, we find
\begin{align*}
\partial_{t}\Delta_{q}\tilde{\Omega} - \mu \Delta\Delta_{q}\tilde{\Omega} = \Delta_{q}\tilde{H} - KV'(t)\Delta_{q}\tilde{\Omega}.
\end{align*}
Hence, we can easily get
\begin{align*}
\frac{d}{dt}e^{-KV(t)}\|\Delta_{q}\tilde{\Omega}(t)\|_{L^{2}} + c 2^{2q}\|\Delta_{q}\tilde{\Omega}(t)\|_{L^{2}} \leq C \|\Delta_{q}\tilde{H}(t)\|_{L^{2}}.
\end{align*}
Noting the definition of hybrid Besov space, we obtain
\begin{align*}
\begin{split}
\frac{d}{dt}e^{-KV(t)}\|\tilde{\Omega}(t)\|_{\tilde{B}^{s-2,s-1}} +
c \|\tilde{\Omega}(t)\|_{\tilde{B}^{s, s+1}}
\leq  C \|\tilde{H}(t)\|_{\tilde{B}^{s-2, s-1}}.
\end{split}
\end{align*}
As in Step 4, denote $2K$ as $C$, we have
\begin{align}\label{omega}
\begin{split}
& \frac{d}{dt}e^{-CV(t)}\|\Omega(t)\|_{\tilde{B}^{s-2,s-1}} +
c e^{-CV(t)}\|\Omega(t)\|_{\tilde{B}^{s,s+1}}    \\
& \quad\quad
\leq C e^{-CV(t)} \|H(t)\|_{\tilde{B}^{s-2,s-1}}.
\end{split}
\end{align}
Denote $\mathcal{E}_{s}(a,d,\Omega) = \mathcal{F}_{s} + \|\Omega(t)\|_{\tilde{B}^{s-2,s-1}}$. Obviously we have
\begin{align*}
\mathcal{E}_{s}(a,d,\Omega) \approx \|a(t)\|_{\tilde{B}^{s-2,s}} + \|d(t)\|_{\tilde{B}^{s-2,s-1}} + \|\Omega(t)\|_{\tilde{B}^{s-2,s-1}}.
\end{align*}
Summing up (\ref{a and d}) and (\ref{omega}), we get (\ref{diff}). Hence, the proof is completed.
\end{proof}
\begin{remark}\label{tuilun1}
For dimension $N = 3$, if we take $s = N/2$, using similar methods, we know that there exists a functional
\begin{align*}
\mathcal{E}_{s}(a, d, \Omega) \approx \|a(t)\|_{\tilde{B}^{N/2-1,N/2}} + \|d(t)\|_{\dot{B}^{N/2-1}} + \|\Omega(t)\|_{\dot{B}^{N/2-1}},
\end{align*}
such that
\begin{align*}
& \frac{d}{dt}\left( e^{-CV(t)} \mathcal{E}_{s}(a,d,\Omega) \right) + ce^{-CV(t)}\Big( \|a(t)\|_{\dot{B}^{N/2}} + \|d(t)\|_{\dot{B}^{N/2+1}}    \\
& \quad\quad
+ \|\Omega(t)\|_{\dot{N/2+1}}\Big) \leq C e^{-CV(t)} \left( \|F(t)\|_{\tilde{B}^{N/2-1,N/2}} + \|G(t), H(t)\|_{\dot{B}^{N/2-1}} \right),
\end{align*}
where $C$ depends on $\mu, \lambda$.
\end{remark}

\section{Energy evolution of negative Besov norms}

In this section, we will derive the evolution of the negative Besov norms of the solution.
The negative Besov space also used in \cite{tanzhong}, however, we derive a different form of
estimates for our low regularity assumption.
\begin{proposition}\label{nBesov}
For $s \in [0, N/2]$, we have
\begin{align*}
\|(a(t), u(t))\|_{\dot{B}_{2,\infty}^{-s}}^{2} \leq e^{\int_{0}^{t} \left( \|a\|_{\dot{B}^{N/2+1}} + \|u\|_{\dot{B}^{N/2+1, N/2+2}} \right) \, d\tau}
\|(a_{0}, u_{0})\|_{\dot{B}_{2,\infty}^{-s}}^{2}.
\end{align*}
\end{proposition}
\begin{proof}
From the systems (\ref{2system}), we obtain
\begin{align*}
& \partial_{t}\Delta_{q}a + \mathrm{div}\Delta_{q}u = -\Delta_{q}(a\mathrm{div}u) - \Delta_{q}(u\cdot\nabla a),   \\
& \partial_{t}\Delta_{q}u - \mu \Delta \Delta_{q}u - (\lambda + \mu)\nabla\mathrm{div}\Delta_{q}u + \nabla\Delta_{q}a   \\
& \quad\quad\quad\quad = - \Delta_{q}(u\cdot\nabla u) - \Delta_{q}\left( h(a)(\mu \Delta u + (\lambda + \mu)\nabla\mathrm{div}u) \right) - \Delta_{q}\left( f(a)\nabla a \right).
\end{align*}
Denoting $(\cdot | \cdot)$ as the $L^{2}$ inner product.
Applying the operator $\Lambda^{-s}$ to both equations and multiplying $\Lambda^{-s}\Delta_{q}a$, $\Lambda^{-s}\Delta_{q}u$ to the above two equations separately, we have
\begin{align*}
(\partial_{t}\Lambda^{-s}\Delta_{q}a | \Lambda^{-s}\Delta_{q}a) & + (\Lambda^{-s}\mathrm{div}\Delta_{q}u | \Lambda^{-s}\Delta_{q}a)   \\
& = - (\Lambda^{-s}\Delta_{q}(a\mathrm{div}u) | \Lambda^{-s}\Delta_{q}a) - (\Lambda^{-s}\Delta_{q}(u\cdot\nabla a) | \Lambda^{-s}\Delta_{q}a),
\end{align*}
and
\begin{align*}
(\partial_{t}\Lambda^{-s}\Delta_{q}u | \Lambda^{-s}\Delta_{q}u) & - \mu (\Lambda^{-s}\Delta\Delta_{q}u | \Lambda^{-s}\Delta_{q}u) -
(\lambda + \mu)(\Lambda^{-s}\nabla\mathrm{div}\Delta_{q}u | \Lambda^{-s}\Delta_{q}u)    \\
& + (\Lambda^{-s}\nabla\Delta_{q}a | \Lambda^{-s}\Delta_{q}u) = -(\Lambda^{-s}\Delta_{q}(u\cdot\nabla u) | \Lambda^{-s}\Delta_{q}u) \\
& - (\Lambda^{-s}\Delta_{q}(h(a)(\mu\Delta u + (\lambda+\mu)\nabla\mathrm{div}u)) | \Lambda^{-s}\Delta_{q}u)    \\
& - (\Lambda^{-s}\Delta_{q}(f(a)\nabla a) | \Lambda^{-s}\Delta_{q}u).
\end{align*}
Summing up the above two equalities, we find
\begin{align}\label{tempest}
\begin{split}
& \frac{1}{2}\frac{d}{dt}\|(\Lambda^{-s}\Delta_{q}a, \Lambda^{-s}\Delta_{q}u)\|_{L^{2}}^{2} + \mu \|\Lambda^{-s}\nabla\Delta_{q}u\|_{L^{2}}^{2}   \\
& \quad\quad\quad\quad\quad\quad\quad\quad + (\lambda + \mu)\|\Lambda^{-s}\mathrm{div}\Delta_{q}u\|_{L^{2}}^{2} = W_{1} + W_{2} + W_{3} + W_{4} + W_{5},
\end{split}
\end{align}
where
\begin{align*}
& W_{1} = - (\Lambda^{-s}\Delta_{q}(a\mathrm{div}u) | \Lambda^{-s}\Delta_{q}a),  \quad W_{2} = - (\Lambda^{-s}\Delta_{q}(u\cdot\nabla a) | \Lambda^{-s}\Delta_{q}a),   \\
& W_{3} = - (\Lambda^{-s}\Delta_{q}(u\cdot\nabla u) | \Lambda^{-s}\Delta_{q}u), \quad W_{4} = - (\Lambda^{-s}\Delta_{q}(f(a)\nabla a) | \Lambda^{-s}\Delta_{q}u),   \\
& W_{5} = - (\Lambda^{-s}\Delta_{q}(h(a)(\mu \Delta u + (\lambda+\mu)\nabla\mathrm{div}u)) | \Lambda^{-s}\Delta_{q}u).
\end{align*}
Now, we give the detailed proof of the estimates about $W_{1}$.
\begin{align}\label{w1pro}
\begin{split}
W_{1} & \leq \|\Lambda^{-s}\Delta_{q}(a\mathrm{div}u)\|_{L^{2}} \|\Lambda^{-s}\Delta_{q}a\|_{L^{2}}   \\
& \leq C \left( \|T_{a}\mathrm{div}u\|_{\dot{B}_{2,\infty}^{-s}} + \|T_{\mathrm{div}u}a\|_{\dot{B}_{2,\infty}^{-s}} + \|R(a, \mathrm{div}u)\|_{\dot{B}_{2,\infty}^{-s}}\right) \|a\|_{\dot{B}_{2,\infty}^{-s}}.
\end{split}
\end{align}
The first term in the bracket can be estimated by using Lemma \ref{inter2} as follows
\begin{align*}
\|T_{a}\mathrm{div}u\|_{\dot{B}_{2,\infty}^{-s}} \leq C \|a\|_{\dot{B}_{2,\infty}^{-s}} \|\mathrm{div}u\|_{\dot{B}^{N/2}}.
\end{align*}
The second term in the bracket can be estimated by using Remark \ref{inter3} as follows
\begin{align*}
\|T_{\mathrm{div}u}a\|_{\dot{B}_{2,\infty}^{-s}} & \leq C \|\mathrm{div}u\|_{\dot{B}_{\infty, 1}^{0}} \|a\|_{\dot{B}_{2,\infty}^{-s}} \\
& \leq C \|\mathrm{div}u\|_{\dot{B}^{N/2}} \|a\|_{\dot{B}_{2,\infty}^{-s}}.
\end{align*}
The third term in the bracket can be estimated by using Lemma \ref{inter4} as follows
\begin{align*}
\|R(a, \mathrm{div}u)\|_{\dot{B}_{2,\infty}^{-s}} \leq C \|a\|_{\dot{B}_{2,\infty}^{-s}} \|\mathrm{div}u\|_{\dot{B}^{N/2}}.
\end{align*}
Combining the above three estimates with (\ref{w1pro}), we find that
\begin{align}\label{w1}
W_{1} \leq C \|a\|_{\dot{B}_{2,\infty}^{-s}}^{2} \|u\|_{\dot{B}^{N/2+1}}.
\end{align}
Similar to the estimate about $W_{1}$, we can get the following estimates
\begin{align*}
& W_{2} \leq C \|a\|_{\dot{B}_{2,\infty}^{-s}} \|u\|_{\dot{B}_{2,\infty}^{-s}} \|a\|_{\dot{B}^{N/2+1}}, \quad W_{3} \leq C \|u\|_{\dot{B}_{2,\infty}^{-s}}^{2}\|u\|_{\dot{B}^{N/2+1}}, \\
& W_{4} \leq C \|a\|_{\dot{B}_{2,\infty}^{-s}}\|u\|_{\dot{B}_{2,\infty}^{-s}}\|a\|_{\dot{B}^{N/2+1}}, \quad W_{5} \leq C \|u\|_{\dot{B}_{2,\infty}^{-s}} \|a\|_{\dot{B}_{2,\infty}^{-s}}
\|u\|_{\dot{B}^{N/2+2}}
\end{align*}
Plugging the above estimates for $W_{1}$ to $W_{5}$ to equation (\ref{tempest}), we will get
\begin{align*}
& \frac{d}{dt}\|(a(t), u(t))\|_{\dot{B}_{2,\infty}^{-s}}^{2} + c \|\nabla u(t)\|_{\dot{B}_{2,\infty}^{-s}}^{2}     \\
& \quad\quad\quad\quad\quad \leq C \left( \|a(t)\|_{\dot{B}^{N/2+1}} + \|u(t)\|_{\dot{B}^{N/2+1, N/2+2}} \right)\|(a(t), u(t))\|_{\dot{B}_{2,\infty}^{-s}}^{2},
\end{align*}
where $s \in [0, N/2]$.
Finally, by Gronwall's inequality, we obtain
\begin{align}\label{nfinal}
\|(a(t), u(t))\|_{\dot{B}_{2,\infty}^{-s}}^{2} \leq e^{\int_{0}^{t} \left( \|a\|_{\dot{B}^{N/2+1}} + \|u\|_{\dot{B}^{N/2+1, N/2+2}} \right) \, d\tau}
\|(a_{0}, u_{0})\|_{\dot{B}_{2,\infty}^{-s}}^{2}.
\end{align}
Hence, the proof is completed.
\end{proof}
\begin{remark}\label{tuilun2}
From Proposition \ref{nBesov}, we see that if we want to control $$\|(a(t), u(t))\|_{\dot{B}_{2,\infty}^{-s}}^{2},$$
we need to control
\begin{align}\label{tuilunjdfk}
\int_{0}^{t} \left( \|a\|_{\dot{B}^{N/2+1}} + \|u\|_{\dot{B}^{N/2+1, N/2+2}} \right) \, d\tau.
\end{align}
When dimensional $N = 3$, if the solution $(\rho, u)$ belongs to the following space
$$\rho-\bar{\rho} \in C([0,+\infty);H^{3}), \quad u \in C([0,+\infty);H^{2})\cap L^{2}(0,T;H^{3}),$$
then we can control the above term (\ref{tuilunjdfk}) by using some basic interpolation inequalities.
\end{remark}

\section{Derive optimal time decay rate}

Firstly, we give the following global well posedness results without proof for it is
similar to \cite{danchin full}.
\begin{theorem}
\label{global}
Suppose that $N \geq 2$. There exists a positive constant $\alpha_{0}$ such that
if $\rho_{0} - 1 \in \tilde{B}^{N/2-1,N/2+1}$, $u_{0} \in \tilde{B}^{N/2-1, N/2}$
with moreover
\begin{align*}
\|\rho_{0} - 1\|_{\tilde{B}^{N/2-1,N/2+1}} + & \|u_{0}\|_{\tilde{B}^{N/2-1, N/2}} \leq \alpha_{0},
\end{align*}
then:
System (\ref{originalsystem}) has a unique global solution $(\rho, u)$ such that
\begin{align*}
& \|\rho - 1\|_{C(\mathbb{R}^{+}, \tilde{B}^{N/2-1, N/2+1})} + \|u\|_{C(\mathbb{R}^{+}, \tilde{B}^{N/2-1, N/2})}  \\
& \quad\quad\quad\quad\quad\quad\quad + \|\rho - 1\|_{L^{1}(\mathbb{R}^{+}, \dot{B}^{N/2+1})} + \|u\|_{L^{1}(\mathbb{R}^{+}, \tilde{B}^{N/2+1, N/2+2})} \leq C \alpha_{0}.
\end{align*}
\end{theorem}

Applying Proposition \ref{hyperbolicparabolic} to the system (\ref{3system}) with $s = \frac{N}{2} + 1$, we will get
\begin{align}\label{inequ}
\begin{split}
& \frac{d}{dt}e^{-CV(t)} \mathcal{E}_{N/2+1}(a, d, \Omega) + c e^{-CV(t)}\left( \|a\|_{\dot{B}^{N/2+1}} + \|(d, \Omega)\|_{\tilde{B}^{N/2+1, N/2+2}} \right)   \\
\leq & C e^{-CV(t)} \left( \|F(t)\|_{\tilde{B}^{N/2-1, N/2+1}} + \|G(t)\|_{\tilde{B}^{N/2-1, N/2}} + \|H(t)\|_{\tilde{B}^{N/2-1, N/2}} \right),
\end{split}
\end{align}
where
\begin{align*}
\mathcal{E}_{N/2+1}(a, d, \Omega) \approx \|a(t)\|_{\tilde{B}^{N/2-1, N/2+1}} + \|u(t)\|_{\tilde{B}^{N/2-1, N/2}}.
\end{align*}
Then, we need to give the estimates about $F$, $G$ and $H$. We will use Lemma \ref{inter2} and Lemma \ref{inter4} frequently without mentioned.
Remaindering that $F = - a \mathrm{div}u$, so we have
\begin{align}\label{estimateF}
\begin{split}
\|F\|_{\tilde{B}^{N/2-1, N/2+1}} & \lesssim \|T_{a}\mathrm{div}u\|_{\tilde{B}^{N/2-1, N/2+1}} + \|T_{\mathrm{div}u}a\|_{\tilde{B}^{N/2-1, N/2+1}} \\
& \quad + \|R(a, \mathrm{div}u)\|_{\tilde{B}^{N/2-1, N/2+1}}   \\
& \lesssim \|a\|_{\tilde{B}^{N/2-1, N/2}}\|\mathrm{div}u\|_{\tilde{B}^{N/2, N/2+1}}     \\
& \quad + \|\mathrm{div}u\|_{\dot{B}^{N/2}}\|a\|_{\tilde{B}^{N/2, N/2+1}}   \\
& \quad + \|a\|_{\tilde{B}^{N/2-1, N/2}}\|\mathrm{div}u\|_{\tilde{B}^{N/2, N/2+1}}  \\
& \leq C \|a\|_{\tilde{B}^{N/2-1, N/2+1}} \|u\|_{\tilde{B}^{N/2+1, N/2+2}}  \\
& \leq C \alpha_{0} \|u\|_{\tilde{B}^{N/2+1, N/2+2}}.
\end{split}
\end{align}
For $G$ and $H$, we give the estimates of main terms as follows
\begin{align}\label{term1}
\begin{split}
\|u\cdot\nabla u\|_{\tilde{B}^{N/2-1, N/2}} & \lesssim \|T_{u}\nabla u\|_{\tilde{B}^{N/2-1,B/2}} + \|T_{\nabla u}u\|_{\tilde{B}^{N/2-1, N/2}}     \\
& \quad + \|R(u, \nabla u)\|_{\tilde{B}^{N/2-1, N/2}}   \\
& \lesssim \|u\|_{\tilde{B}^{N/2-1, N/2}}\|\nabla u\|_{\dot{B}^{N/2}} + \|\nabla u\|_{\dot{B}^{N/2}}\|u\|_{\tilde{B}^{N/2-1, N/2}}      \\
& \quad + \|u\|_{\tilde{B}^{N/2-1, N/2}}\|\nabla u\|_{\dot{B}^{N/2}}      \\
& \leq C \|u\|_{\tilde{B}^{N/2-1, N/2}}\|u\|_{\tilde{B}^{N/2+1, N/2+2}} \\
& \leq C \alpha_{0} \|u\|_{\tilde{B}^{N/2+1, N/2+2}},
\end{split}
\end{align}
\begin{align}\label{term2}
\begin{split}
\|h(a)\Delta u\|_{\tilde{B}^{N/2-1, N/2}} & \leq C \|a\|_{\dot{B}^{N/2}} \|\Delta u\|_{\tilde{B}^{N/2-1, N/2}}  \\
& \leq C \|a\|_{\tilde{B}^{N/2-1, N/2+1}} \|u\|_{\tilde{B}^{N/2+1, N/2+2}}      \\
& \leq C \alpha_{0} \|u\|_{\tilde{B}^{N/2+1, N/2+2}},
\end{split}
\end{align}
\begin{align}\label{term3}
\begin{split}
\|f(a)\nabla a\|_{\tilde{B}^{N/2-1, N/2}} & \leq C \|a\|_{\tilde{B}^{N/2-1, N/2}}\|\nabla a\|_{\dot{B}^{N/2}}   \\
& \leq C \|a\|_{\tilde{B}^{N/2-1, N/2+1}} \|a\|_{\dot{B}^{N/2+1}}     \\
& \leq C \alpha_{0} \|a\|_{\dot{B}^{N/2+1}}.
\end{split}
\end{align}
Combining estimates (\ref{term1}), (\ref{term2}) and (\ref{term3}), we get the estimates for $G$ and $H$ as follows
\begin{align}\label{estimateG}
\|G\|_{\tilde{B}^{N/2-1, N/2}} \leq C \alpha_{0} \left( \|a\|_{\dot{B}^{N/2+1}} + \|u\|_{\tilde{B}^{N/2+1, N/2+2}} \right),
\end{align}
\begin{align}\label{estimateH}
\|H\|_{\tilde{B}^{N/2-1, N/2}} \leq C \alpha_{0} \left( \|a\|_{\dot{B}^{N/2+1}} + \|u\|_{\tilde{B}^{N/2+1, N/2+2}} \right).
\end{align}
Plugging (\ref{estimateF}), (\ref{estimateG}) and (\ref{estimateH}) into (\ref{inequ}), we find
\begin{align}\label{plugFGH}
\begin{split}
\frac{d}{dt}e^{-CV(t)}\mathcal{E}_{N/2+1}(a, d, \Omega) & + c e^{-CV(t)} \left( \|a\|_{\dot{B}^{N/2+1}} + \|u\|_{\tilde{B}^{N/2+1, N/2+2}} \right)  \\
& \leq  C \alpha_{0} e^{-CV(t)} \left( \|a\|_{\dot{B}^{N/2+1}} + \|u\|_{\tilde{B}^{N/2+1, N/2+2}} \right).
\end{split}
\end{align}
Denoting $\mathcal{E}_{N/2+1}(a,d,\Omega)$ simply as $\mathcal{E}_{N/2+1}(t)$, and
taking $\alpha_{0}$ small enough, we finally get
\begin{align}\label{final}
\begin{split}
\frac{d}{dt}e^{-CV(t)}\mathcal{E}_{N/2+1}(t) + c e^{-CV(t)}\left( \|a\|_{\dot{B}^{N/2+1}} + \|u\|_{\tilde{B}^{N/2+1, N/2+2}} \right) \leq 0.
\end{split}
\end{align}
Finally, we need some interpolation estimates to close our proof.
From the definition of hybrid Besov space, we know that
\begin{align*}
& \|a\|_{\tilde{B}^{N/2-1, N/2+1}} \approx \|a\|_{\dot{B}^{N/2-1}} + \|a\|_{\dot{B}^{N/2+1}}, \\
& \|u\|_{\tilde{B}^{N/2-1, N/2+1}} \approx \|u\|_{\dot{B}^{N/2-1}} + \|u\|_{\dot{B}^{N/2+1}}, \\
& \|u\|_{\tilde{B}^{N/2+1, N/2+2}} \approx \|u\|_{\dot{B}^{N/2+1}} + \|u\|_{\dot{B}^{N/2+2}}.
\end{align*}
By Lemma \ref{inter1}, we have
\begin{align*}
\|a(t)\|_{\dot{B}^{N/2-1}} \leq C \|a(t)\|_{\dot{B}_{2,\infty}^{-s}}^{\theta_{1}} \|a(t)\|_{\dot{B}^{N/2+1}}^{1-\theta_{1}},
\end{align*}
where $\theta_{1} = \frac{2}{N/2+1+s}$.
Similarly, we also have
\begin{align*}
& \|u(t)\|_{\dot{B}^{N/2-1}} \leq C \|u(t)\|_{\dot{B}_{2,\infty}^{-s}}^{\theta_{1}} \|u(t)\|_{\dot{B}^{N/2+1}}^{1-\theta_{1}},      \\
& \|u(t)\|_{\dot{B}^{N/2}} \leq C \|u(t)\|_{\dot{B}_{2,\infty}^{-s}}^{\theta_{2}} \|u(t)\|_{\dot{B}^{N/2+2}}^{1-\theta_{2}},
\end{align*}
where $\theta_{2} = \frac{2}{N/2+2+s}$.
Combining the results in Proposition \ref{nBesov} and Theorem \ref{global}, we find
\begin{align*}
\|a\|_{L^{\infty}(\mathbb{R}^{+}, \dot{B}_{2,\infty}^{-s})} + \|u\|_{L^{\infty}(\mathbb{R}^{+}, \dot{B}_{2,\infty}^{-s})} \leq C
\left( \|a_{0}\|_{\dot{B}_{2,\infty}^{-s}} + \|u_{0}\|_{\dot{B}_{2,\infty}^{-s}} \right).
\end{align*}
So we easily obtain
\begin{align}\label{interori}
\begin{split}
& \|a(t)\|_{\dot{B}^{N/2+1}} \geq C \|a(t)\|_{\dot{B}^{N/2-1}}^{\frac{1}{1-\theta_{1}}} = C \|a(t)\|_{\dot{B}^{N/2-1}}^{1 + \frac{2}{N/2+s-1}},   \\
& \|u(t)\|_{\dot{B}^{N/2+1}} \geq C \|u(t)\|_{\dot{B}^{N/2-1}}^{\frac{1}{1-\theta_{1}}} = C \|u(t)\|_{\dot{B}^{N/2-1}}^{1 + \frac{2}{N/2+s-1}},   \\
& \|u(t)\|_{\dot{B}^{N/2+2}} \geq C \|u(t)\|_{\dot{B}^{N/2}}^{\frac{1}{1-\theta_{2}}} = C \|u(t)\|_{\dot{B}^{N/2}}^{1 + \frac{2}{N/2+s}}.
\end{split}
\end{align}
By taking $\alpha_{0}$ small enough in Theorem \ref{global}, we can assume
\begin{align*}
\|a(t)\|_{\tilde{B}^{N/2-1, N/2+1}} \leq 1, \quad \|u(t)\|_{\tilde{B}^{N/2-1, N/2}} \leq 1.
\end{align*}
Notice that $1 + \frac{2}{N/2+s-1}  > 1 + \frac{2}{N/2+s}$, we find
\begin{align}\label{interfinal}
\begin{split}
& \|a(t)\|_{\dot{B}^{N/2+1}} \geq C \|a(t)\|_{\dot{B}^{N/2-1}}^{1 + \frac{2}{N/2+s-1}}, \quad \|a(t)\|_{\dot{B}^{N/2+1}} \geq C \|a(t)\|_{\dot{B}^{N/2+1}}^{1 + \frac{2}{N/2+s-1}},   \\
& \|u(t)\|_{\dot{B}^{N/2+1}} \geq C \|u(t)\|_{\dot{B}^{N/2-1}}^{1 + \frac{2}{N/2+s-1}}, \quad \|u(t)\|_{\dot{B}^{N/2+2}} \geq C \|u(t)\|_{\dot{B}^{N/2}}^{1 + \frac{2}{N/2+s-1}}.
\end{split}
\end{align}
Plugging (\ref{interfinal}) into (\ref{final}), we obtain
\begin{align*}
& \frac{d}{dt}e^{-CV(t)}\mathcal{E}_{N/2+1}(t)    \\
& \quad\quad\quad\quad + c \left( e^{-CV(t)}\left( \|a(t)\|_{\tilde{B}^{N/2-1, N/2+1}} +
\|u(t)\|_{\tilde{B}^{N/2-1, N/2}} \right) \right)^{1 + \frac{2}{N/2+s-1}} \leq 0.
\end{align*}
Since $\mathcal{E}_{N/2+1}(t) \approx  \|a(t)\|_{\tilde{B}^{N/2-1, N/2+1}} + \|u(t)\|_{\tilde{B}^{N/2-1, N/2}}$, we finally get
\begin{align}\label{finalineq}
\begin{split}
\frac{d}{dt}e^{-CV(t)}\mathcal{E}_{N/2+1}(t) + c \left( e^{-CV(t)}\mathcal{E}_{N/2+1}(t) \right)^{1+\frac{2}{N/2+s-1}} \leq 0.
\end{split}
\end{align}
Solving this differential inequality, we could obtain
\begin{align*}
\mathcal{E}_{N/2+1}(t) \leq C e^{CV(t)} \left( \mathcal{E}_{N/2+1}(0)^{-\frac{2}{N/2+s-1}} + \frac{2C}{N/2+s-1} t \right)^{-\frac{N/2+s-1}{2}}
\end{align*}
From Theorem \ref{global}, we know that $V(t)$ is bounded by the initial data, there exists a constant $C$ such that
\begin{align}\label{Linf}
\|a(t)\|_{\tilde{B}^{N/2-1, N/2+1}} + \|u(t)\|_{\tilde{B}^{N/2-1, N/2}} \leq C (1+t)^{-\frac{N/2-1+s}{2}}.
\end{align}
Due to the relationship between Sobolev space and Besov space, we obtain
\begin{align}\label{decay}
\begin{split}
& \|\Lambda^{N/2-1} a(t)\|_{L^{2}} + \|\Lambda^{N/2-1} u(t)\|_{L^{2}} \\
\leq & C \left( \|a(t)\|_{\dot{B}^{N/2-1}} + \|u(t)\|_{\dot{B}^{N/2-1}} \right)    \\
\leq & C \left( \|a(t)\|_{\tilde{B}^{N/2-1, N/2+1}} + \|u(t)\|_{\tilde{B}^{N/2-1, N/2}} \right)   \\
\leq & C (1+t)^{-\frac{N/2-1+s}{2}}.
\end{split}
\end{align}
For $\ell \in (-s, N/2-1)$, we have
\begin{align*}
\|a(t)\|_{\dot{B}^{\ell}} \leq \|a(t)\|_{\dot{B}_{2,\infty}^{-s}}^{\theta} \|a(t)\|_{\dot{B}^{N/2-1}}^{1-\theta},
\end{align*}
where $\theta = \frac{N/2-1-\ell}{N/2-1+s}$.
By (\ref{decay}), we then obtain
\begin{align*}
\|a(t)\|_{\dot{B}^{\ell}} & \leq \|a(t)\|_{\dot{B}_{2,\infty}^{-s}}^{\frac{N/2-1-\ell}{N/2-1+s}} \|a(t)\|_{\dot{B}^{N/2-1}}^{1-\frac{N/2-1-\ell}{N/2-1+s}}       \\
& \leq C (1+t)^{-\frac{\ell+s}{2}}.
\end{align*}
Similar to the above analysis for $a(t)$, we also have the following estimates for $u(t)$
\begin{align*}
\|u(t)\|_{\dot{B}^{\ell}} \leq C (1+t)^{-\frac{\ell+s}{2}}.
\end{align*}
At this stage, we complete the proof of Theorem \ref{mainth}.

\begin{remark}\label{tuilun3}
In order to get the Theorem \ref{tuilun}, we need the following theorem which is proved in \cite{danchin NS}.
\begin{theorem}
Suppose that $N \geq 2$. There exists a positive constant $\alpha_{0}$ such that
if $\rho_{0} - 1 \in \tilde{B}^{N/2-1,N/2}$, $u_{0} \in \dot{B}^{N/2-1}$
with moreover
\begin{align*}
\|\rho_{0} - 1\|_{\tilde{B}^{N/2-1,N/2}} + & \|u_{0}\|_{\dot{B}^{N/2-1}} \leq \alpha_{0},
\end{align*}
then:
System (\ref{originalsystem}) has a unique global solution $(\rho, u)$ such that
\begin{align*}
& \|\rho - 1\|_{C(\mathbb{R}^{+}, \tilde{B}^{N/2-1, N/2})} + \|u\|_{C(\mathbb{R}^{+}, \dot{B}^{N/2-1})}  \\
& \quad\quad\quad\quad\quad\quad\quad + \|\rho - 1\|_{L^{1}(\mathbb{R}^{+}, \dot{B}^{N/2})} + \|u\|_{L^{1}(\mathbb{R}^{+}, \dot{B}^{N/2+1})} \leq C \alpha_{0}.
\end{align*}
\end{theorem}

With this theorem, Remark \ref{tuilun1}, Remark \ref{tuilun2} and the result in paper \cite{zhang2014}, we can mimic the procedure
of the last section to get our Theorem \ref{tuilun}
\end{remark}

\section{Acknowledgements}

This research is support partially by National Natural Science Foundation of China under the grant no. 11131006, and by the National Basic Research Program of China under the grant no. 2013CB329404.
J. Jia would like to thank China Scholarship Council that has provided a scholarship for his research work in the United States.

\end{document}